\documentclass[12pt]{article}

\usepackage{graphicx}
\usepackage{verbatim}
\usepackage{textcomp}
\usepackage{amsmath,amssymb,amsthm}
\usepackage{esint}
\usepackage{cite}
\usepackage{setspace}
\newenvironment{proof1}
		{{\sc Proof of Theorem~\ref{main1}.}}
	{{\sc q.e.d.} \\}
	
\newenvironment{proof2}
		{{\sc Proof of Theorem~\ref{main2}.}}
	{{\sc q.e.d.} \\}

\newcommand{\abs}[1]{\left\vert#1\right\vert}
\newcommand{\R}{\mathbb R}
\newcommand{\tu}{\tilde{u}}
\newcommand{\ra}{\rangle}
\newcommand{\la}{\langle}

\AtEndDocument{\footnotesize
	\textsc{Department of Mathematics, UBC-PIMS}\\
	\textsc{121-1984 Mathematics Road}\\
	\textsc{Vancouver, B.C., Canada, V6T 1Z2}\\
	\textit{E-mail address}: \texttt{bfreidin@math.ubc.ca}

	\textsc{Yau Mathematical Sciences Center, Tsinghua University}\\
	\textsc{Jin Chun Yuan West Building 308, Tsinghua University}\\
	\textsc{ Beijing, China, 100084}\\
	\textit{E-mail address}: \texttt{yingyzhang@tsinghua.edu.cn}
}

\newtheorem{theorem}{Theorem}
\newtheorem{lemma}[theorem]{Lemma}
\newtheorem{proposition}[theorem]{Proposition}

\theoremstyle{definition}
\newtheorem{remark}[theorem]{Remark}

\begin{document}

\begin{center}
\Large{A Liouville-type theorem and Bochner formula for harmonic maps into metric spaces}
\end{center}

\vspace*{0.05in}

\begin{center}
Brian Freidin and Yingying Zhang
\end{center}

\vspace*{0.1in}

\begin{quote} 
\emph{Abstract.} We study analytic properties of harmonic maps from Riemannian polyhedra into CAT($\kappa$) spaces for $\kappa\in\{0,1\}$. Locally, on each top-dimensional face of the domain, this amounts to studying harmonic maps from smooth domains into CAT($\kappa$) spaces. We compute a target variation formula that captures the curvature bound in the target, and use it to prove an $L^p$ Liouville-type theorem for harmonic maps from admissible polyhedra into convex CAT($\kappa$) spaces. Another consequence we derive from the target variation formula is the Eells-Sampson Bochner formula for CAT(1) targets.
\let\thefootnote\relax\footnote{2010 Mathematics Subject Classification 53C43}
 \end{quote}

\section{Introduction}

The study of harmonic maps between and into singular spaces has received a lot of attention since the early '90s. The study began in \cite{gromov-schoen} when Gromov and Schoen developed harmonic map theory into simplicial complexes in order to study non-archimedian superrigidity. In \cite{korevaar-schoen} Korevaar and Schoen constructed harmonic maps into arbitrary CAT(0) targets as a boundary value problem. The study was further extended to domains of Riemannian polyhedra, e.g. in (\cite{chen},\cite{daskal-mese1},\cite{daskal-mese2},\cite{eells-fuglede}), as well as arbitrary metric measure spaces, e.g. in (\cite{jost1},\cite{jost2}).

For a more complete survey of the theory of harmonic maps between singular spaces, we refer the reader to the above references. For our purposes, we will consider as domains admissible Riemannian polyhedra. By this we mean connected, dimensionally homogeneous simplicial complexes with a Riemannian metric on each top-dimensional face, such that the complement of the codimension two skeleton is still connected. More details about such spaces and their applications can be found in  \cite{eells-fuglede}.

The study of harmonic maps into singular spaces was further extended to targets with positive curvature bounds. In his thesis, Serbinowski \cite{serbinowski} established existence, uniqueness, and regularity results for harmonic maps from Riemannian manifolds into CAT(1) spaces under Dirichlet boundary conditions. \cite{bfhmsz2} proved the existence of harmonic maps from surfaces into CAT(1) spaces for the homotopy problem and therefore generalized the classical Sacks-Uhlenbeck theorem \cite{sacks-uhlenbeck}, while \cite{bfhmsz1} established that harmonic maps from Riemannian polyhedra to CAT(1) spaces are H\"older continuous, and Lipschitz away from the $n-2$ skeleton of the domain.

There are many known Liouville-type theorems for harmonic functions and harmonic maps. For a detailed history we refer the reader to the introduction of \cite{kuwae-sturm}. In \cite{yau1}, Yau shows that any positive harmonic function on a complete manifold of non-negative Ricci curvatures must be constant. In \cite{cheng}, Cheng generalizes Yau's gradient estimate to harmonic maps into non-positively curved manifolds and obtains the Liouville theorem for harmonic maps, and in \cite{choi} Choi generalizes the result further to targets with posive curvature bounds. Recently, a gradient estimate in the case of singular targets has been carried out in \cite{zzz}, resulting in a Liouville-type theorem.

Our first goal is to prove a Liouville-type theorem for harmonic maps from Riemannian polyhedra into CAT(1) spaces. The main analytic result generalizes a result of \cite{yau2}, where Yau shows that on a complete manifold, any non-negative subharmonic function in $L^p$ ($p>1$) must be constant. In Section~\ref{Sliouville}, under the assumption that that image of the harmonic map is sufficiently small, the distance function on the image will be convex, and we can compute a target variation formula. One result of this formula is that the distance squared to a point, $d^2(u,Q)$ is a subharmonic function as long as $d(u,Q)\le\frac{\pi}{2}$. Combining these two results yields a Liouville-type theorem.

\begin{theorem}\label{main1}
Let $(X,g)$ be a complete admissible Riemannian polyhedron, and let $u:(X,g)\to(Y,d)$ be a harmonic map into a CAT(1) space. If there is a point $Q\in Y$ so that $d(u,Q)<\frac{\pi}{2}$ for all $x\in M$ and $d^2(u,Q)\in L^1(X)$, then $u$ is a constant map. If $(Y,d)$ is CAT(0) and $d^2(u,Q)\in L^p(X)$ for some $p>1$, we need not assume $d(u,Q)$ is bounded and still $u$ is a constant map.
\end{theorem}

In Section~\ref{Sbochner} we will apply the target variation from Section~\ref{Sliouville} and the methods of \cite{freidin} to extend the Bochner formula found in \cite{freidin} to the CAT(1) setting. Let $\pi$ denote the pull-back tensor as defined by Korevaar and Schoen in \cite{korevaar-schoen} as the symmetric tensor polarizing the quadratic form
\[
\pi(v,v) = \abs{u_*v}^2 = \lim_{\epsilon\to 0}\frac{d^2(u(x+\epsilon v),u(x))}{\epsilon^2}.
\]
This tensor satisfies $tr(\pi) = \abs{\nabla u}^2$. With this definition, we can state out Bochner formula for maps into CAT(1) spaces.
\begin{theorem}\label{main2}
Let $u:(M,g)\to(Y, d)$ be a harmonic map from a Riemannian manifold $(M,g)$ into a CAT(1) space $(Y, d)$, and suppose $d(u,Q)<\frac{\pi}{2}$ for some $Q\in Y$. Then $\abs{\nabla u}^2$ satisfies the weak differential inequality
\[
\frac{1}{2}\Delta\abs{\nabla u}^2 \ge \la Ric,\pi\ra + \abs{\pi}^2 - \abs{\nabla u}^4,
\]
where $Ric$ is the Ricci tensor on $M$.
\end{theorem}

\begin{remark}
Compare the above theorem to the results of \cite{freidin}. When the target $Y$ is a CAT($\kappa$) metric space, for $\kappa\in\{0,-1\}$, a harmonic map $u:M\to Y$ satisfies
\[
\frac{1}{2}\Delta\abs{\nabla u}^2 = \la Ric,\pi\ra - \kappa\left(\abs{\nabla u}^4 - \abs{\pi}^2\right).
\]
Together with the above result for $\kappa = 1$, these results entirely capture the roles of domain and target curvature from the original Eells-Sampson Bochner formula, though it still lacks the Hessian term from that original formula.
\end{remark}

\section{A Liouville-type theorem}\label{Sliouville}

The main analytic tool in our approach to proving the $L^p$ Liouville-type theorem will be to apply a reverse Poincar\'e inequality, much in the style of Schoen-Yau \cite{schoen-yau} or of \cite{sinaei}. Here $(X,g)$ will be a complete admissible Riemannian polyhedron.

\begin{proposition}\label{Lp liouville}
If $f\in L^p(X)$ for $p>1$ is a locally Lipschitz, non-negative, subharmonic function, then $f$ is a constant.
\end{proposition}

\begin{proof}
Since $f$ is subharmonic, for any non-negative test function $\phi$ we have $\int_X\la\nabla\phi,\nabla f\ra\le0$. Choose $\phi$ so that $\phi\equiv 1$ on a ball $B_R$, $\phi\equiv 0$ on $X\backslash B_{2R}$, and $\abs{\nabla\phi}\le\frac{1}{R}$. For instance, $\phi(r) = \max\{0,\min\{1,2-\frac{r}{R}\}\}$ will do. Then for $p>1$,
\begin{eqnarray*}
0 & \ge & \int_X \la\nabla (\phi^2 f^{p-1}), \nabla f\ra\\
 & = & \int_X \la\nabla (\phi^2)f^{p-1} + (p-1)f^{p-2}	\phi^2\nabla f,\nabla f\ra\\
 & = & 2\int_X f^{p-1}\phi\la \nabla\phi,\nabla f\ra + (p-1)\int_Xf^{p-2}\phi^2\abs{\nabla f}^2\\
 & \ge & (p-1)\int_X f^{p-2}\phi^2\abs{\nabla f}^2 - 2\left(\int_X f^p\abs{\nabla\phi}^2\right)^\frac{1}{2} \left( \int_X f^{p-2}\phi^2\abs{\nabla f}^2\right)^\frac{1}{2},
\end{eqnarray*}
i.e.
\[
4\int_X f^p\abs{\nabla\phi}^2 \ge (p-1)^2	\int_X f^{p-2}\phi^2\abs{\nabla f}^2.
\]

Using our assumption on $\phi$, we deduce
\[
\int_{B_R} f^{p-2}\abs{\nabla f}^2 \le \frac{C}{R^2}\int_{B_{2R}}f^p.
\]
Now send $R\to\infty$, the integral on the right hand side remains bounded, as $f\in L^p$, while $\frac{C}{R^2}\to0$. Hence $\nabla f\equiv 0$. By the admissibility of $X$, this implies that $f$ is a constant function.
\end{proof}





The main geometric tool we will use is a target variation formula that will give us a positive subharmonic function.

\begin{proposition}\label{target variation}
For an energy minimizing map $u:(M,g)\to B_{\frac{\pi}{2}}(Q)\subset(Y,d_Y)$ from a Riemannian manifold $(M,g)$ to a convex neighborhood in a CAT(1) space $(Y,d_Y)$, we have in the weak sense

\[
\frac{1}{2} \Delta d_Y^2(u(x),Q) \ge \abs{\nabla u}^2 + \left(\frac{d_Y(u,Q)\cos d_Y(u,Q)}{\sin d_Y(u,Q)} - 1\right)\left(\abs{\nabla u}^2 - \abs{\nabla d_Y(u,Q)}^2\right).
\]
\end{proposition}

\begin{proof}

Choose normal coordinates on $\mathbb{S}^2$. This idenitifies $\mathbb{S}^2\backslash\{p\}$ with $B_{\pi}\subset\R^2$, but with the metric $ds^2 = dr^2 + sin^2(r)d\theta^2$. In particular it identifies a hemisphere with $B_\frac{\pi}{2}\subset\R^2$.

For the geodesic triangle with vertices $\{Q, u(x), u(x+\epsilon v)\}$ in $Y$,  we consider the comparison triangle with vertices $\{0, \tu(x), \tu(x+\epsilon v)\} $  in $B_\frac{\pi}{2}$ (with the spherical metric).
 Then we extend the map $\tu$ by mapping the geodesic joining $x$ and $x+\epsilon v$ to the spherical geodesic joining $\tu(x)$ and $\tu(x+\epsilon v)$.

Choose $\eta$ as a non-negative smooth test function with compact support in a neighborhood of  $x$, define the map $\tu_t(y) = (1-t\eta(y))\tu(y)$ for $y$ between $x$ and $x+\epsilon v$. Its derivative at $x$ in the direction of $v$ is

\[
v\cdot\nabla \tu_t(x) = (1-t\eta(x))\big(v\cdot\nabla \tu(x)\big) - t\big(v\cdot\nabla\eta(x)\big)\tu(x).
\]

This vector is based at $\tu_t(x)$, so to calculate its magnitude we must use the metric at that point. The vector $\tu(x)$ is already pointing in the radial direction when it is based at $\tu(x)$ or even at $\tu_t(x)$.  For $v\cdot\nabla\tu(x)$, the radial component is simply $v\cdot\nabla\abs{\tu}_0(x)$. The tangential component, when measured at $\tu(x)$, has norm
\[
\sqrt{\abs{v\cdot\nabla\tu(x)}^2_{\tu(x)} - \abs{v\cdot\nabla\abs{\tu}_0(x)}^2}.
\] 
Here subscripts are used to denote the point at which the norm is being taken. In particular, $\abs{p}_0$ is the distance of the point $p$ from the point $0$ on $\mathbb{S}^2$. For a vector $w\perp\frac{\partial}{\partial r}$, $\abs{w}_{\tu(x)} = \frac{\sin\abs{\tu(x)}_0}{\abs{\tu(x)}_0}\abs{w}_0$, and $\abs{w}_{\tu_t(X)} = \frac{\sin\abs{\tu_t(x)}_0}{\abs{\tu_t(x)}_0}\abs{w}_0 = \frac{\sin\abs{\tu_t(x)}_0}{(1-t\eta(x))\abs{\tu(x)}_0}\abs{w}_0$. Hence the length of the tangential component of $v\cdot\nabla\tu(x)$, when measured at $\tu_t(x)$, is
\[
\frac{\sin\abs{\tu_t(x)}_0}{(1-t\eta(x))\sin\abs{\tu(x)}_0}\sqrt{\abs{v\cdot\nabla\tu(x)}^2_{\tu(x)} - \abs{v\cdot\nabla\abs{\tu}_0(x)}^2}.
\]

Now
\begin{eqnarray*}
\abs{v\cdot\nabla\tu_t(x)}^2_{\tu_t(x)} & = & \Big[ (1-t\eta(x))(v\cdot\nabla\abs{\tu}_0(x)) - t(v\cdot\nabla\eta(x))\abs{\tu(x)}_0\Big]^2\\
 & & + \frac{\sin^2\abs{\tu_t(x)}_0}{\sin^2\abs{\tu(x)}_0}\left(\abs{v\cdot\nabla\tu(x)}^2_{\tu(x)} - \abs{v\cdot\nabla\abs{\tu}_0(x)}^2\right).
\end{eqnarray*}
And by the Taylor's expansion,  we see that
\[
\frac{\sin^2\abs{\tu_t(x)}_0}{\sin^2\abs{\tu(x)}_0} = 1 - 2t\eta\frac{\abs{\tu(x)}_0\cos\abs{\tu(x)}_0}{\sin\abs{\tu(x)}_0} + O(t^2).
\]
So now we have

\begin{eqnarray*}
\abs{v\cdot\nabla\tu_t(x)}^2_{\tu_t(x)} & = & (1-2t\eta(x))\abs{v\cdot\nabla\tu(x)}^2_{\tu(x)} - t(v\cdot\nabla\abs{\tu}^2_0(x))(v\cdot\nabla\eta(x)) \\
 & & - 2t\eta(x)\left(\frac{\abs{\tu(x)}_0\cos\abs{\tu(x)}_0}{\sin\abs{\tu(x)}_0}-1\right)\left(\abs{v\cdot\nabla\tu(x)}^2_{\tu(x)} - \abs{v\cdot\nabla\abs{\tu}_0(x)}^2\right)+ O(t^2).
\end{eqnarray*}


Now, by the triangle comparison property, and send $\epsilon\to 0$, we have
\begin{eqnarray*}
\abs{(u_t)_*v(x)}^2  &\le& 
  (1-2t\eta(x))\abs{u_* v(x)}^2 - t(v\cdot\nabla\eta(x))(v\cdot\nabla d_Y^2(u,Q)(x)) \\
 &&  - 2t\eta(x) \left( \frac{d_Y(u,Q)\cos d_Y(u,Q)}{\sin d_Y(u,Q)} - 1 \right) \left( \abs{u_*v(x)}^2 - (v\cdot\nabla d_Y(u,Q))^2(x) \right)+ O(t^2).
\end{eqnarray*}

Averaging over the sphere of unit normal vectors yields

\begin{eqnarray*}
\abs{\nabla u_t}^2 & \le & (1-2t\eta)\abs{\nabla u}^2 - t\nabla\eta\cdot\nabla d_Y^2(u,Q)\\
 & & - 2t\eta\left( \frac{d_Y(u,Q)\cos d_Y(u,Q)}{\sin d_Y(u,Q)} - 1\right) \left( \abs{\nabla u}^2 - \abs{\nabla d_Y(u,Q)}^2\right) + O(t^2).
\end{eqnarray*}


By the energy minimizing assumption of $u$, as $t\to 0$, we see
\[
0 \le \int_M\left[-2\eta\abs{\nabla u}^2 +d_Y^2(u,Q)\Delta\eta - 2\eta\left(\frac{d_Y(u,Q)\cos d_Y(u,Q)}{\sin d_Y(u,Q)}-1\right)\left(\abs{\nabla u}^2 - \abs{\nabla d_Y(u,Q)}^2\right)\right].
\]
In other words, we have the weak inequality
\[
\frac{1}{2}\Delta d_Y^2(u,Q) \ge \abs{\nabla u}^2 + \left(\frac{d_Y(u,Q)\cos d_Y(u,Q)}{\sin d_Y(u,Q)} - 1\right)\left( \abs{\nabla u}^2 - \abs{\nabla d_Y(u,Q)}^2\right).
\]
\end{proof}

Even though Proposition~\ref{target variation} is stated when the domain of the harmonic map is a manifold, it can be applied on the top dimensional faces of a Riemannian polyhedron. To be precise, for a polyhedron $X$ with faces $F_i$, $d_Y^2(u,Q)$ satisfies the weak inequality on each face, and by the admissibility, 
\[
\int_X\la\nabla d_Y^2(u,Q),\nabla\eta\ra = \sum_i\int_{F_i}\la\nabla d_Y^2(u,Q),\nabla\eta\ra.
\]
And so if the right hand side of the expression for $\Delta d_Y^2(u,Q)$ is positive then $d_Y^2(u,Q)$ is subharmonic in the sense that is required for Proposition~\ref{Lp liouville}. 

\medskip

Now we are ready to prove the $L^p$ Liouville-type theorem.

\begin{proof1}
In the case of a CAT(1) target, apply Proposition~\ref{target variation} on each face of $X$. Since $0\le\frac{z\cos z}{\sin z}\le 1$ for $0\le z\le\frac{\pi}{2}$, we see that $d_Y^2(u,Q)$ is subharmonic on $X$ in the sense required for Proposition~\ref{Lp liouville}. Then, since $d_Y^2(u,Q)\in L^\infty\cap L^1$, it is in $L^p$ for all $1<p<\infty$ by interpolation. Hence by Proposition~\ref{Lp liouville},  $d_Y^2(u,Q)$ is constant. And

\[
0 = \frac{1}{2}\Delta d_Y^2(u,Q) \ge \frac{d_Y(u,Q)\cos d_Y(u,Q)}{\sin d_Y(u,Q)}\abs{\nabla u}^2 \ge 0.
\]

Since $\frac{z\cos z}{\sin z}$ vanishes first at $\frac{\pi}{2}$, and $d_Y(u,Q)<\frac{\pi}{2}$ by assumption, $\abs{\nabla u}^2$ must vanish. In other words, $u$ is a constant map.

In the case where the target is CAT(0),  we may use the target variation formula from \cite{gromov-schoen} or from \cite{freidin}, which says
\[
\frac{1}{2}\Delta d_Y^2(u,Q) \ge \abs{\nabla u}^2 \ge 0.
\]
The remainder of the argument, applying Proposition~\ref{Lp liouville} follows exactly the same. Once $\Delta d_Y^2(u,Q) = 0$, we conclude again that $\abs{\nabla u}^2\equiv 0$ and  $u$ is a constant map.
\end{proof1}

\section{Bochner Formula}\label{Sbochner}

In \cite{freidin}, the first author obtained a version of the Eells-Sampson Bochner formula  for the case of NPC spaces. The main technique is based on the target variation formula and monotonicity of frequency function. In this section, we will derive a Bochner formula for the case of CAT(1) spaces by a similar argument. We will use $\la-,-\ra$ and $\abs{\cdot}$ to denote the inner product structure on the space of symmetric $2$-tensors, so that $\la A,B\ra = g^{ik}g^{j\ell}A_{ij}B_{k\ell}$. We first have the following lemma which will be used in deriving a mean value inequality.

\begin{lemma}\label{grad d squared}
For a Lipschitz map $u:(M,g)\to B_{\frac{\pi}{2}}(Q)\subset (Y,d)$ into a convex neighborhood in a CAT(1) space,
\[
\int_{B_\sigma}\abs{\nabla d^2(u,Q)}^2d\mu = \frac{4\omega_n}{n+2}\abs{\pi}^2(0)\sigma^{n+2} + o(\sigma^{n+2}).
\]
Here $B_\sigma\subset M$ is a ball of radius $\sigma$, $\abs{\pi}^2$ means $g^{ik}g^{j\ell}\pi_{ij}\pi_{k\ell}$, 
  $\omega_n$ is volume of the unit ball. Around points where $\abs{\nabla u}^2\neq 0$, we can also write above as
\[
\int_{B_\sigma}\abs{\nabla d^2(u,Q)}^2d\mu = 4\left[\frac{\abs{\pi}^2(0)}{(n+2)\abs{\nabla u}^2(0)}\sigma^2 + o(\sigma^2)\right]\int_{B_\sigma}\abs{\nabla u}^2d\mu.
\]
\end{lemma}

\begin{proof}
For a point $x\in B_\sigma$ and a unit vector $v\in T_{x}M$, we consider a quadrilateral with vertices $u(0)=Q$, $u(x)$, $u(x+\epsilon v)$, and $u((1-\epsilon)(x+\epsilon v))$ in $Y$. We are going to construct a comparison quadrilateral in $\mathbb{S}^2$. Choosing normal coordinates identifies $\mathbb{S}^2\backslash\{p\}$ with $B_\frac{\pi}{2}\subset\mathbb{R}^2$, but with the metric $ds^2 = dr^2 + sin^2(r)d\theta^2$. The Euclidean metric is given by $dr^2+r^2d\theta^2$, and $r^2-\sin^2(r) = O(r^4)$. Since all distances involved are at most $O(\sigma)$, the metric we take on $B_\frac{\pi}{2}$ is $(1+O(\sigma^2))(dx^2+dy^2)$. We'll work for the rest of this argument with the Euclidean metric $dx^2+dy^2$, and the $O(\sigma^2)$ difference will be negligible.

We  first construct a subembedding of the four points $Q = u(0)$, $u(x)$, $u(x+\epsilon v)$, and $u((1-\epsilon)(x+\epsilon v))$. That is, four points $0$, $\tu(x)$, $\tu(x+\epsilon v)$, and $\tu((1-\epsilon)(x+\epsilon v))$ in $\R^2$ so that $d(u(x+\epsilon v),Q)\le \abs{\tu(x+\epsilon v)}$ and $d(u(x),u((1-\epsilon)(x+\epsilon v)))\le\abs{\tu(x)-\tu((1-\epsilon)(x+\epsilon v))}$, while the other four corresponding distances are equal.  We then extend $\tu$ to be an affine map into $\R^2$, so that
\[
\tu(tx+s\epsilon v) = t\tu(x) + s\big(\tu(x+\epsilon v)-\tu(x)\big) + \frac{t-1}{\epsilon}\big((1-\epsilon)\tu(x+\epsilon v) - \tu((1-\epsilon)(x+\epsilon v))\big).
\]

Now we compute:
\begin{eqnarray*}
v\cdot\nabla d^2(u,Q)(x) & = & \lim_{s\to0}\frac{d^2(u(x+s\epsilon v),Q)-d^2(u(x),Q)}{s\epsilon}\\
 & \le & \lim_{s\to0}\frac{\abs{\tu(x+s\epsilon v)}^2-\abs{\tu(x)}^2}{s\epsilon}\\
 & = & 2\tu(x)\cdot\frac{\tu(x+\epsilon v)-\tu(x)}{\epsilon}\\
 & = & 2(x\cdot\nabla\tu(x))\cdot(v\cdot\nabla\tu(x))\\
 & = & \frac{1}{1-\epsilon}\bigg(\abs{x\cdot\nabla\tu(x)}^2 + (1-\epsilon)^2\abs{v\cdot\nabla\tu(x)}^2\\
 & & - \abs{((1-\epsilon)v-x)\cdot\nabla\tu(x)}^2\bigg).\\
\abs{x\cdot\nabla\tu(x)}^2 & = & \abs{\tu(x)}^2\\
 & = & d^2(u(x),Q)\\
 & = & \pi_{ij}(x)x^ix^j + e(x),
 \end{eqnarray*}
where $e(x)$ is the error term, and $\frac{e(x)}{|x|^2}$ can be bounded by a constant  only depending on the energy of the map and domain geometry.

 \begin{eqnarray*}
\abs{v\cdot\nabla\tu(x)}^2 & = & \frac{\abs{\tu(x+\epsilon v)-\tu(x)}^2}{\epsilon^2} + o(1)\\
 & = & \abs{u_*v}^2(x) + o(1).\\
\abs{((1-\epsilon)v-x)\cdot\nabla\tu(x)}^2 & = & \frac{\abs{\tu((1-\epsilon)(x+\epsilon v)) - \tu(x)}^2}{\epsilon^2} + o(1)\\
 & \ge & \abs{u_*((1-\epsilon)v-x)}^2(x) + o(1)\\
 & = & \abs{u_*(v-x)}^2(x) + o(1).
\end{eqnarray*}
Putting all these together, and letting $\epsilon\to 0$, we get
\begin{eqnarray*}
v\cdot\nabla d^2(u,Q)(x) & \le & d^2(u(x),Q) + \abs{u_*v}^2(x) - \abs{u_*(v-x)}^2(x)\\
 & = & 2\pi_{ij}(x)x^iv^j + e(x).
\end{eqnarray*}

If we consider the opposite direction $-v$, the above inequality still holds and it yields 

\[
v\cdot\nabla d^2(u,Q)(x) = 2\pi_{ij}(x)x^iv^j + e(x).
\]
Here we remark that the sign on $e(x)$ does not matter, as this term will be negligibly small upon integration. 

Averaging over the unit sphere of direction vectors at $x$ yields
\[
\abs{\nabla d^2(u,Q)}^2(x) = 4\pi_{ik}(x)\pi_{jk}(x)x^ix^j + e(x).
\]
And integrating over the ball $B_\sigma$ yields
\[
\int_{B_\sigma}\abs{\nabla d^2(u,Q)}^2 = \frac{4\omega_n}{n+2}\abs{\pi}^2(0)\sigma^{n+2} + o(\sigma^{n+2}).
\]
\end{proof}

The remaining computations involve in the monotonicity of frequency function, it is first considered in \cite{gromov-schoen} to study the Lipschitz regularity for harmonic map into a NPC complex. In \cite{freidin}, the first author explores such monotonicity and applies it to obtain a Bochner formula for NPC targets. Since only domain variation will be needed in the following argument,  we outline the sequence of results obtained in\cite{freidin} here for completeness. 

Let $u:(M,g)\to B_{\pi/2}(Q)\subset(Y,d)$ be a harmonic map into a convex neighborhood in some CAT(1) space $(Y,d)$. We will now show the mean value inequality:
\begin{lemma}\label{mean value}
\[
\fint_{B_\sigma(x_0)}\abs{\nabla u}^2d\mu \ge \abs{\nabla u}^2(x_0) + \frac{\la Ric,\pi\ra(x_0) + \abs{\pi}^2(x_0) - \abs{\nabla u}^4(x_0)}{n+2}\sigma^2 + o(\sigma^2).
\]
\end{lemma}

\begin{proof}
Let all tensors in the following equations be evaluated at the point $x_0$. For most of these computations, assume $\abs{\nabla u}^2(x_0)\neq0$. Throughout we will use the following functions:
\begin{eqnarray*}
E(\sigma) & = & \int_{B_\sigma}\abs{\nabla u}^2d\mu\\
I(\sigma) & = & \int_{\partial B_\sigma}d^2(u,Q)\,d\Sigma\\
A(\sigma) & = & \int_{\partial B_\sigma}\abs{\frac{\partial u}{\partial r}}^2d\Sigma.
\end{eqnarray*}

Rearrange the statement of Proposition~\ref{target variation} to see
\[
\abs{\nabla u}^2 \le \tan d(u,Q)\Delta d(u,Q) + \abs{\nabla d(u,Q)}^2.
\]
Then integrate this inequality over $B_\sigma$, and use the divergence theorem, the Cauchy-Schwarz inequality, and the triangle inequality (in the form of $\big|\frac{\partial}{\partial r}d(u,Q)\big|\le\abs{\frac{\partial u}{\partial r}}$), to see
\[
E(\sigma) \le \left(A(\sigma)\int_{\partial B_\sigma}\tan^2d(u,Q)d\Sigma\right)^\frac{1}{2} - \int_{B_\sigma}\tan^2d(u,Q)\abs{\nabla d(u,Q)}^2d\mu.
\]

Using the Taylor expansion $\tan^2(t) = t^2 + \frac{2}{3}t^4 + O(t^6)$, and integrating $d^4(u,Q) = \pi_{ij}\pi_{k\ell}x^ix^jx^kx^l + o(\sigma^4)$, we get the following comparison:

\[
\int_{\partial B_\sigma}\tan^2d(u,Q)d\Sigma = \left[1 + \frac{2(2\abs{\pi}^2 + \abs{\nabla u}^4)}{3(n+2)\abs{\nabla u}^2}\sigma^2 + o(\sigma^2)\right]I(\sigma).
\]

This Taylor expansion for $\tan^2(t)$ also lets us replace the expression $\int_{B_\sigma}\tan^2d(u,Q)\abs{\nabla d(u,Q)}^2d\mu$ with $\int_{B_\sigma}d^2(u,Q)\abs{\nabla d(u,Q)}^2d\mu$ for the cost of $O(\sigma^{n+4})$. And we have an expression for $\int_{B_\sigma}d^2(u,Q)\abs{\nabla d(u,Q)}^2d\mu$ from Lemma~\ref{grad d squared}. Hence we have
\begin{eqnarray*}
E(\sigma) & \le & \left[1 + \frac{2\abs{\pi}^2 + \abs{\nabla u}^4}{3(n+2)\abs{\nabla u}^2}\sigma^2 + o(\sigma^2)\right]\Big(A(\sigma)I(\sigma)\Big)^\frac{1}{2}\\
 & & - \left[\frac{\abs{\pi}^2}{(n+2)\abs{\nabla u}^2}\sigma^2 + o(\sigma^2)\right]E(\sigma).
\end{eqnarray*}
Isolating $E(\sigma)$, we find
\begin{equation}\label{EAI}
E(\sigma) \le \left[1 + \frac{\abs{\nabla u}^4 - \abs{\pi}^2}{3(n+2)\abs{\nabla u}^2}\sigma^2 + o(\sigma^2)\right]\Big(A(\sigma)I(\sigma)\Big)^\frac{1}{2}.
\end{equation}

The domain variation formula from \cite{gromov-schoen} was expanded in \cite{freidin} to include the leading order growth near $\sigma = 0$. It says
\begin{align}
(2-n) & \int_{B_\sigma(x_0)} \abs{\nabla u}^2 d\mu + \sigma\int_{\partial B_\sigma(x_0)} \left[ \abs{\nabla u}^2 - 2 \abs{\frac{\partial u}{\partial r}}^2\right] d\Sigma\nonumber\\
 & = \frac{\omega_n\left(2\la Ric,\pi\ra(x_0) - S(x_0)\abs{\nabla u}^2(x_0)\right)}{3(n+2)}\sigma^{n+2} + o(\sigma^{n+2}).\label{domain}
\end{align}
Using this formula, the logarithmic derivative of $E(\sigma)$ is
\begin{eqnarray*}
\frac{E'(\sigma)}{E(\sigma)} & = & \frac{1}{E(\sigma)}\int_{\partial B_\sigma}\abs{\nabla u}^2d\mu\\
 & = & \frac{n-2}{\sigma} + \frac{2A(\sigma)}{E(\sigma)} + \frac{2\la Ric,\pi\ra - S\abs{\nabla u}^2}{3(n+2)\abs{\nabla u}^2}\sigma + o(\sigma).
\end{eqnarray*}

Using the expansion $\sqrt{g}(x) = 1-R_{ij}(0)x^ix^j + O(\abs{x}^3)$ for the volume element in normal coordinates, the logarithmic derivative of $I(\sigma)$ becomes
\begin{eqnarray*}
\frac{I'(\sigma)}{I(\sigma)} & = & \frac{1}{I(\sigma)}\int_{\partial B_\sigma}\frac{\partial}{\partial r}d^2(u,Q)d\Sigma +\frac{n-1}{\sigma}+O(\sigma^2)\\
 & & - \frac{1}{3\sigma I(\sigma)}\int_{\partial B_\sigma}d^2(u,Q)Ric_0(x,x)dS\\
 & = & \frac{n-1}{\sigma} +  \frac{1}{I(\sigma)}\int_{\partial B_\sigma}\frac{\partial}{\partial r}d^2(u,Q)d\Sigma\\
 & & - \frac{2\la Ric,\pi\ra + S\abs{\nabla u}^2}{3(n+2)\abs{\nabla u}^2}\sigma + o(\sigma),
\end{eqnarray*}
where $S$ is the scalar curvature of $M$.

Combining, we get the monotonicity formula from \cite{gromov-schoen}, with the asymptotic behavior near $\sigma = 0$:
\begin{eqnarray}
\frac{d}{d\sigma}\log\left(\frac{\sigma E(\sigma)}{I(\sigma)}\right) & = & \frac{2A(\sigma)}{E(\sigma)} - \frac{1}{I(\sigma)}\int_{\partial B_\sigma}\frac{\partial}{\partial r}d^2(u,Q)d\Sigma + \frac{4\la Ric,\pi\ra}{3(n+2)\abs{\nabla u}^2}\sigma + o(\sigma)\nonumber\\
 & \ge & \frac{2A(\sigma)}{E(\sigma)} - \frac{2}{I(\sigma)}\Big(A(\sigma)I(\sigma)\Big)^\frac{1}{2} + \frac{4\la Ric,\pi\ra}{3(n+2)\abs{\nabla u}^2}\sigma + o(\sigma)\nonumber\\
 & \ge & 2\left[1+\frac{\abs{\pi}^2-\abs{\nabla u}^4}{3(n+2)\abs{\nabla u}^2}\sigma^2 + o(\sigma^2)\right]\sqrt{\frac{A(\sigma)}{I(\sigma)}} - 2\sqrt{\frac{A(\sigma)}{I(\sigma)}}\nonumber\\
 & & + \frac{4\la Ric,\pi\ra}{3(n+2)\abs{\nabla u}^2}\sigma + o(\sigma)\nonumber\\
 & = & \frac{4\la Ric,\pi\ra + 2\abs{\pi}^2 - 2\abs{\nabla u}^4}{3(n+2)\abs{\nabla u}^2}\sigma + o(\sigma).\nonumber\\
\frac{\sigma E(\sigma)}{I(\sigma)} & \ge & 1 + \frac{2\la Ric,\pi\ra + \abs{\pi}^2 - \abs{\nabla u}^4}{3(n+2)\abs{\nabla u}^2}\sigma^2 + o(\sigma^2).\label{order}
\end{eqnarray}
Here we have used the facts that $\frac{A(\sigma)}{I(\sigma)} = \sigma^2 + o(\sigma^2)$ and that $\lim_{\sigma\to0}\frac{\sigma E(\sigma)}{I(\sigma)} = 1$ at points where $\abs{\nabla u}^2\neq 0$.

Now combining \eqref{EAI} and \eqref{order}, we have
\begin{eqnarray}
\sigma\Big(A(\sigma)I(\sigma)\Big)^\frac{1}{2} & \ge & \left[1 + \frac{\abs{\pi}^2 - \abs{\nabla u}^4}{3(n+2)\abs{\nabla u}^2}\sigma^2 + o(\sigma^2)\right]\sigma E(\sigma)\nonumber\\
 & \ge & \left[1 + \frac{2\la Ric,\pi\ra + 2\abs{\pi}^2 - 2\abs{\nabla u}^4}{3(n+2)\abs{\nabla u}^2}\sigma^2 + o(\sigma^2)\right]I(\sigma).\nonumber\\
\sigma A(\sigma) & \ge & \left[1 + \frac{2\la Ric,\pi\ra + 2\abs{\pi}^2 - 2\abs{\nabla u}^4}{3(n+2)\abs{\nabla u}^2}\sigma^2 + o(\sigma^2)\right]\Big(A(\sigma)I(\sigma)\Big)^\frac{1}{2}\nonumber\\
 & \ge & \left[1 + \frac{2\la Ric,\pi\ra + 3\abs{\pi}^2 - 3\abs{\nabla u}^4}{3(n+2)\abs{\nabla u}^2}\sigma^2 + o(\sigma^2)\right]E(\sigma)\label{A vs I}
\end{eqnarray}

Finally, to prove the mean value inequality, we first observe that it is trivially true at points where $\abs{\nabla u}^2 = 0$. At the other points, the above calculations allow us to manipulate the derivative of the mean value of $\abs{\nabla u}^2$. We will write $V_\sigma$ for $Vol(B_\sigma(0))$, and first we need the Bishop-Gromov volume comparison, \cite[Proposition 6]{freidin}, which states
\begin{equation}\label{bishop}
Vol(\partial B_\sigma) = \left(\frac{n}{\sigma} - \frac{S}{3(n+2)}\sigma + O(\sigma^2)\right)V_\sigma.
\end{equation}
Now we compute:
\begin{align*}
\frac{d}{d\sigma}\fint_{B_\sigma}\abs{\nabla u}^2d\mu & = \frac{1}{V_\sigma^2}\left(V_\sigma\int_{\partial B_\sigma}\abs{\nabla u}^2d\Sigma - Vol(\partial B_\sigma)E(\sigma)\right) & \\
 & = \frac{2}{\sigma V_\sigma}\left[\sigma A(\sigma) + \left(-1+\frac{\la Ric,\pi\ra}{3(n+2)\abs{\nabla u}^2}\sigma^2 + o(\sigma^2)\right)E(\sigma)\right] & \text{by \eqref{domain} and \eqref{bishop}}\\
 & \ge \frac{2}{\sigma V_\sigma}\left(\frac{\la Ric,\pi\ra + \abs{\pi}^2 - \abs{\nabla u}^4}{(n+2)\abs{\nabla u}^2}\sigma^2 + o(\sigma^2)\right)E(\sigma) & \text{by \eqref{A vs I}}\\
\frac{d}{d\sigma}\log\fint_{B_\sigma}\abs{\nabla u}^2d\mu & \ge 2\frac{\la Ric,\pi\ra + \abs{\pi}^2 - \abs{\nabla u}^2}{(n+2)\abs{\nabla u}^2}\sigma + o(\sigma) & \\
\fint_{B_\sigma}\abs{\nabla u}^2d\mu & \ge \abs{\nabla u}^2 + \frac{\la Ric,\pi\ra + \abs{\pi}^2 - \abs{\nabla u}^2}{n+2}\sigma^2 + o(\sigma^2) & 
\end{align*}
\end{proof}

Finally we can deduce from the mean value inequality the Bochner formula.

\begin{proof2}
In \cite[Proposition 23]{freidin}, it is shown that if $f\in L^\infty$ satisfies for almost every $x_0$ the integral inequality
\[
\fint_{B_\sigma(x_0)}fd\mu \ge f(x_0) + \phi(x_0)\sigma^2 + o(\sigma^2),
\]
then $f$ satisfies the weak differential inequality
\[
\frac{1}{2}\Delta f \ge (n+2)\phi.
\]
Apply this to the integral inequality of Lemma~\ref{mean value} to conclude
\[
\frac{1}{2}\Delta\abs{\nabla u}^2 \ge \la Ric,\pi\ra + \abs{\pi}^2 - \abs{\nabla u}^4.
\]
\end{proof2}

\section*{Acknowledgements} The work contained here composed a part of the first author's thesis at Brown University. The authors would like to thank Chikako Mese for her generous support and useful discussions. The first author would like to thank Christine Breiner for encouragement and support, and the second author would like to thank Huai-Dong Cao and Xiaofeng Sun for their continual encouragement.

\end{document}